\documentclass{ws-m3as}
\usepackage[utf8]{inputenc}

\usepackage{amsmath,amscd,amssymb}
\usepackage{enumerate}
\usepackage{indentfirst}
\usepackage{latexsym}
\usepackage{multicol}
\usepackage{verbatim}
\usepackage{color}
\usepackage{parallel}
\usepackage{cases}
\usepackage{array}
\usepackage{booktabs}
\usepackage{framed}

\usepackage[dvipsnames]{xcolor}
\usepackage{tikz-cd}
\usetikzlibrary{calc,arrows,positioning}
\usepackage{pgfplots}
\pgfplotsset{compat=1.14}

\numberwithin{equation}{section}

\newcommand{\R}{{\mathbb R}}

\def\eps{\eps}

\def\cN{{\mathcal N}}

\def\A{{\mathcal A}}

\newcommand{\LL}{\mathbf L}
\newcommand{\PP}{\mathbf \Phi}
\newcommand{\DD}{\mathbf D}
\newcommand{\xx}{\mathbf x}
\newcommand{\vv}{\mathbf v}

\newtheorem{defi}{Definition}
\newtheorem{lem}{Lemma}

\newtheorem{prop}{Proposition}
\newtheorem{theo}{Theorem}
\newtheorem{rem}{Remark}

\def\dd{\mathrm d}
\def\dt{\Delta t}
\usepackage{MnSymbol}

\begin{document}
 
\title{Density-induced Consensus Protocol}

\author{Piotr Minakowski}

\address{Institute of Analysis and Numerics, Otto-von-Guericke
    University Magdeburg, \\ Universit\"atsplatz 2, 39106 Magdeburg,
    Germany,\\
    piotr.minakowski@ovgu.de}

\author{Piotr B. Mucha}
\address{Institute of Applied Mathematics and Mechanics, University of Warsaw, \\ ul. Banacha 2, 02-097 Warszawa, Poland \\ 
    p.mucha@mimuw.edu.pl }

\author{Jan Peszek\footnote{Corresponding author}}
\address{Institute of Applied Mathematics and Mechanics, University of Warsaw, \\ ul. Banacha 2, 02-097 Warszawa, Poland \\ 
    j.peszek@mimuw.edu.pl }

\maketitle

\begin{history}
	\received{(Day Month Year)}
	\revised{(Day Month Year)}
	\comby{(xxxxxxxxxx)}
\end{history}

\begin{abstract}
The paper introduces a model of collective behavior where agents receive 
information only from sufficiently dense crowds in their immediate vicinity.
The system is an asymmetric, density-induced version of the
Cucker-Smale model with short-range interactions. 
We prove the basic mathematical properties of the system and concentrate 
on the presentation of interesting behaviors of the solutions. 
The results are illustrated by numerical simulations.
\end{abstract}


\ccode{AMS Subject Classification: }

\section{Introduction}

We consider an ensemble of $N$ agents with $(x_i(t),v_i(t))\in \R^{2d}$ denoting the 
opinion/position and the tendency/velocity of $i$th agent at the time $t\geq 0$. 
The agents follow the density-induced consensus protocol (DI)

\begin{subnumcases}{\label{eq1}}
    \dot x_i=v_i, & $x_i(0) = x_{i0}\in\R^d$, \label{eq1a}\\
    \dot v_i=\displaystyle \sum_{k\in\cN_i}M(v_k-v_i), & $v_i(0) = v_{i0}\in\R^d$, \label{eq1b}
\end{subnumcases}

\noindent
where $\cN_i$ is a neighbor set of $i$th agent defined through the following relation: given positive parameters $\delta, m$ and  $h$, for $t\geq h$ we define
\begin{equation}\label{nor}
\begin{split}
k\in \cN_i(t)\ \Leftrightarrow \ & x_k(t-h)\in B(x_i(t-h),\delta)\ \mbox{and}\ \\ & \#\Big\{k\in\{1,...,N\}: x_k(t-h)\in B(x_i(t-h),\delta)\Big\} > m
\end{split}
\end{equation}
where $B(x_i(t-h),\delta)$ is an open ball centered at $x_i(t-h)$ with radius $\delta$. For $t\in [0,h)$ we take
\begin{equation}\label{nor1}
\begin{split}
k\in \cN_i(t)\ \Leftrightarrow \ & x_k(0)\in B(x_i(0),\delta)\ \mbox{and}\ \\ & \#\Big\{k\in\{1,...,N\}: x_k(0)\in B(x_i(0),\delta)\Big\} > m.
\end{split}
\end{equation}

Parameter $0<h$, negligibly small compared to the rest of parameters, is introduced to ensure that the neighbour  sets $\cN_i(t)$ are well defined.   Indeed, 
taking $h=0$ causes instability of the system if a particle is situated at the boundary of $B(x_i(t),\delta)$.
A natural interpretation of $h$ is a time step from a discrete in time version of (\ref{eq1}). To grasp the 
intuition behind the behaviour, we can view $h$ as $0$ and focus on the qualitative analysis of 
the model, which is the main goal of the paper. Further explanation can be found in Remark \ref{Rh}.
Parameter $M = M(i,N,\#\cN_i)$ is a normalizing factor 
discussed later; to fix our attention we may assume that $M = \kappa/\#\cN_i$, where $\kappa>0$ is 
the non-dimensional coupling strength and $\#\cN_i$ is the number of elements in~$\cN_i$.

Condition \eqref{nor} introduces a two-step verification 
of whether $k$ belongs to the set of neighbors of $i$. First, $x_k(t-h)$ is required to be close 
to $x_i(t-h)$. Second, the crowd density in the immediate vicinity of $i$ needs to be
large enough so that the number of the individuals with $x_k(t-h)\in B(x_i(t-h),\delta)$ is larger than $m$. If the second condition is not satisfied, 
the set of neighbors of $i$ is empty. Observe that the rightmost condition in definition \eqref{nor} is 
asymmetric and consequently so is the relation of adjacency $\rightsquigarrow$ defined by 
$k\rightsquigarrow i \Leftrightarrow k\in\cN_i$. 

Dividing the right-hand side of \eqref{nor} by $\delta^d$ (where $d$ is the dimension of the space) identifies $\frac{m}{\delta^d}$ as a threshold imposed on the empirical density of the particles. Hence, the interaction is induced by sufficiently high density of the agents and the DI protocol operates within the following general paradigm: 

\begin{framed}
 \begin{center}
{\it To influence the opinion of an individual,\\ communication with a sufficiently dense nearby crowd is required}.
  \end{center}
\end{framed}

 \noindent

The protocol is specific to real-life phenomena related to societal dynamics, where individuals do
not interact with separate agents but are highly susceptible to the influence of crowds. It is inspired by such phenomena as emergence of trends in decision-making and viral videos in social media.


Formally \eqref{eq1} is a second order system which in the kinetic formalism determines
the acceleration of particles in terms of their positions and velocities. 
Indeed, it originates from a highly recognizable model of collective behavior, 
the classical Cucker-Smale (CS) system
\begin{subnumcases}{\label{cs}}
    \dot x_i=v_i, & $x_i(0) = x_{i0}\in\R^d$, \label{csa}\\
    \dot v_i=\displaystyle \frac{1}{N}\displaystyle\sum_{k=1}^N\psi(|x_i-x_k|)(v_k-v_i), & $v_i(0) = v_{i0}\in\R^d$, \label{csb}
\end{subnumcases}
with $\psi(s) = (1+s)^{-\alpha}$. 

Comparing the DI and CS models, we observe {\it the first 
prominent feature} of the DI protocol: it leads to the emergence of sharply distinguishable, dense
clusters. Moreover, we observe local flocking, with nontrivial dependence on 
the density; particularly the $time=150$ velocity of singletons varies significantly,  see  Fig.~\ref{fig:1}.

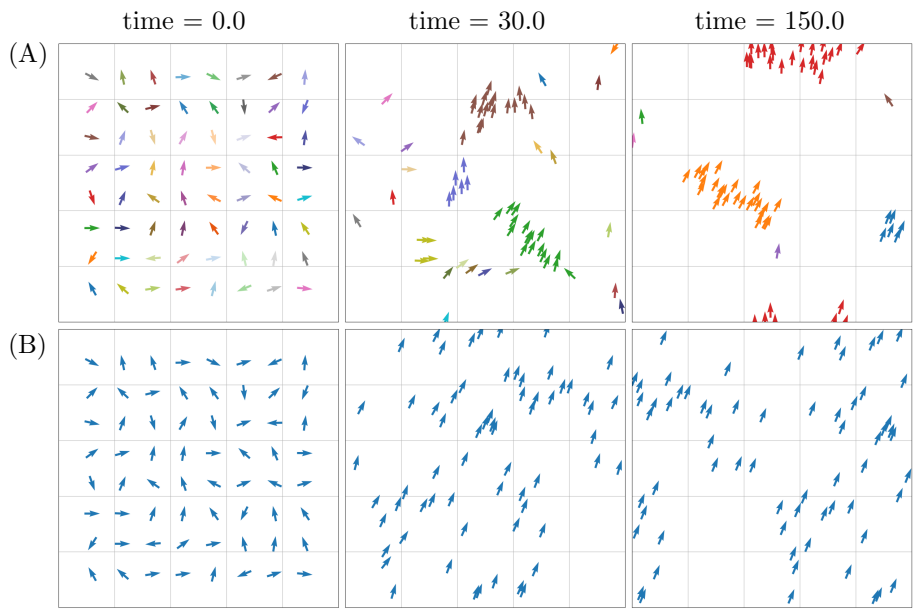
\begin{figure}[h!]
    \setlength{\unitlength}{.1\textwidth}
    \begin{picture}(10,6)
    \put(.5,0){\includegraphics[width=.9\textwidth]{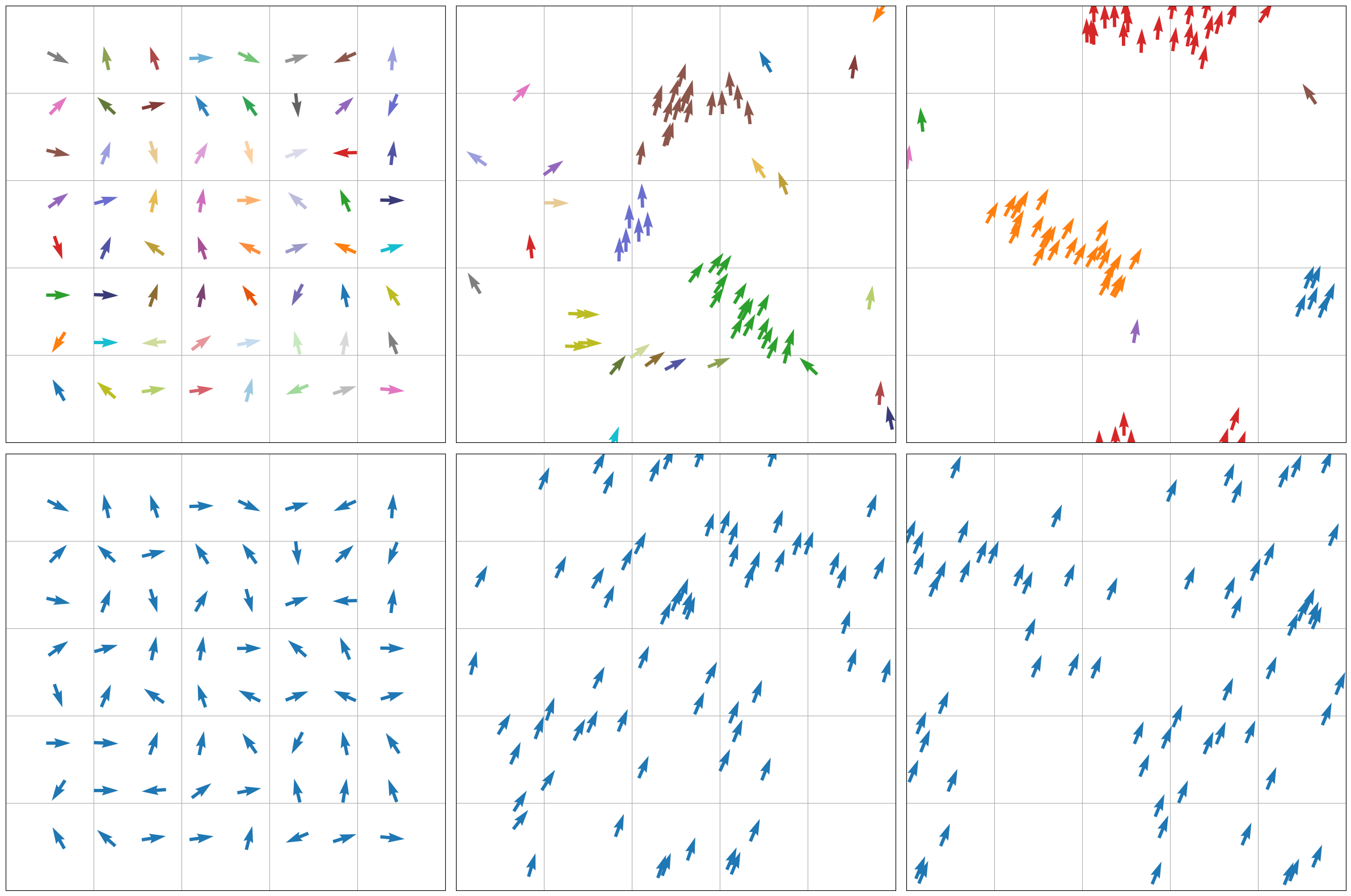}}
    \put(0,5.75){(A)}
    \put(0,2.7){(B)}
    
    \put(1.2,6.1){time = 0.0}
    \put(4.2,6.1){time = 30.0}
    \put(7.2,6.1){time = 150.0}
    
    \end{picture}
    
    \caption{Behaviour of particles for DI (A) and CS (B). Identical initial data lead to sharper 
    clusters for the DI model. The color coding represents clusters of indirect communication - multiple clusters for DI and a single cluster for (global) CS.}
    \label{fig:1}
\end{figure}

{\it The second prominent feature} of the DI model is a structural asymmetry of the interactions 
showcased in Fig. \ref{fig:2}.
In the case (A), at $time=2.0$ the orange singleton already influences the cluster but the cluster 
does not influence the singleton yet, because the number of the particles in the vicinity of the singleton is insufficient 
for it to receive communication.  
Consequently, there is a period of time when the singleton influences the cluster but is yet to be 
affected by the cluster, which results in a particularly strong asymmetry of the interaction. On the 
other hand in the case (B), at $time=2.0$ the orange
singleton is already influenced by the cluster. Therefore, a singleton can change the direction of an 
entire crowd and whether it succeeds depends on the spacial distribution of the crowd itself.
These phenomena are generally impossible to obtain in most well-known models of consensus.

\begin{figure}[h!]
    \setlength{\unitlength}{.1\textwidth}
    \begin{picture}(10,4)
    \put(.5,0){\includegraphics[width=.9\textwidth]{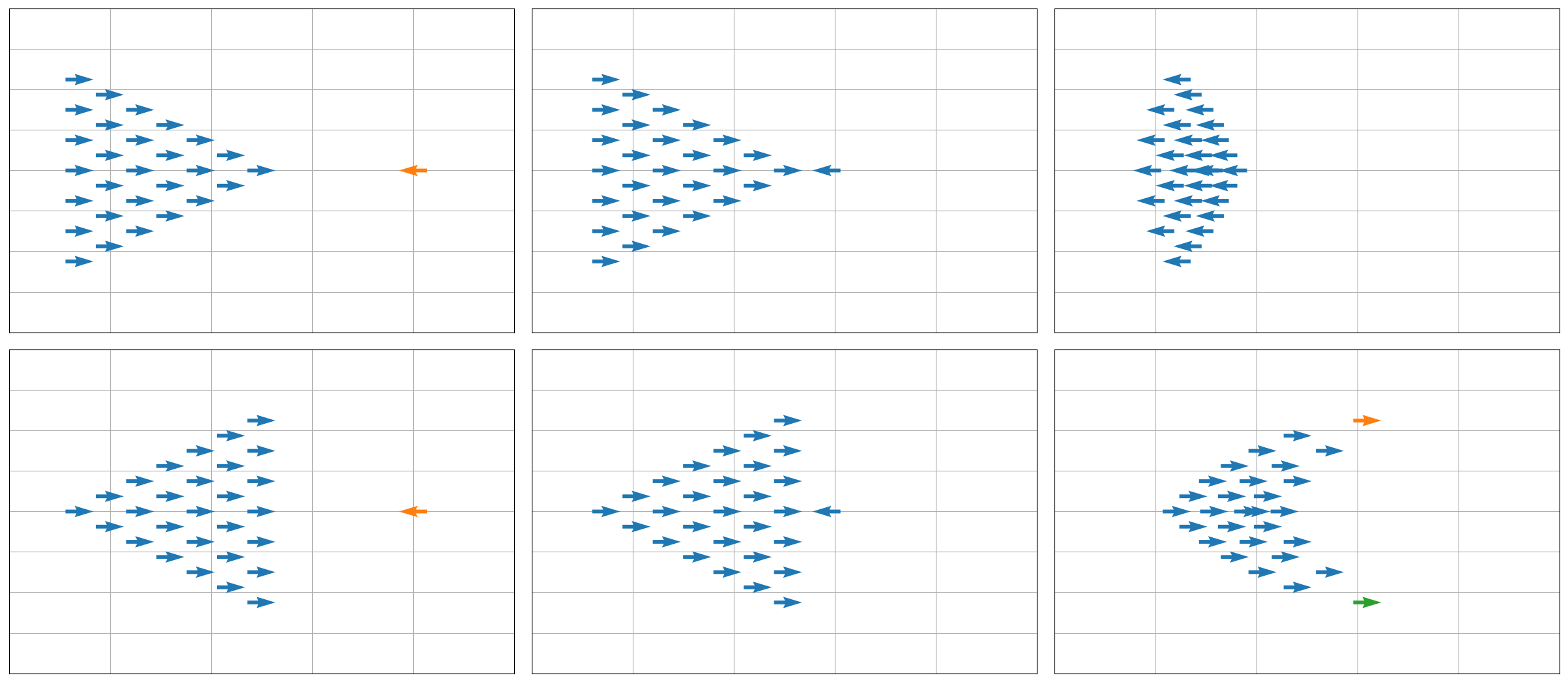}}
    \put(0,3.6){(A)}
    \put(0,1.7){(B)}
    \put(1.25,4){time = 0.0}
    \put(4.25,4.){time = 2.0}
    \put(7.25,4.){time = 30.0}
    \end{picture}
    
    \caption{Interaction between a group and an individual for the DI model. Initially (A) and (B) differ only by the shape of the cluster. In (A) the individual diverts the cluster while in (B) the cluster diverts the individual.}
    \label{fig:2}
\end{figure}
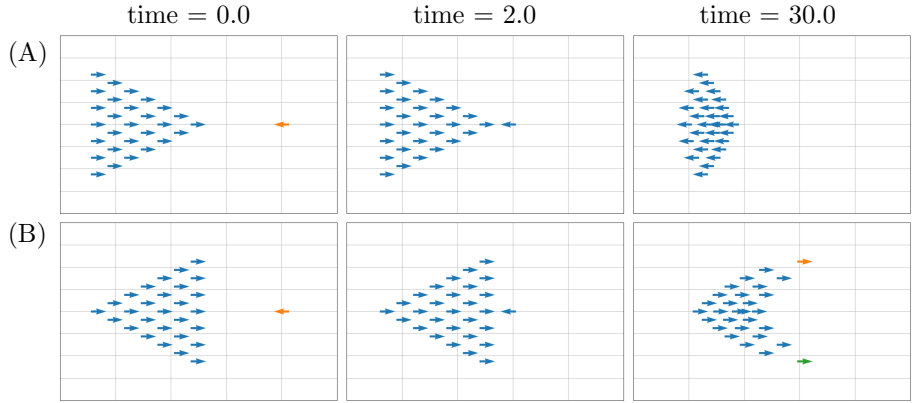

{\bf Notation.} In what follows we use bold symbols for points in $\R^{dN}$ or $\R^{d\# \A}$ space, where a~cluster $\A$ is any subset of $\{1,...,N\}$, namely
\begin{align*}
\xx = (x_{1},...,x_{N})\in \R^{dN}, \quad \vv = (v_1,...,v_N)\in \R^{dN},
 \quad \xx_\A = (x_{i_1},...,x_{i_{\# \A}})\in \R^{d\# \A}.
\end{align*}  
This convention applies for instance to the initial data  $(\xx_0,\vv_0)\in \R^{2dN}$.

Moreover, we introduce averages,  together with their restriction to clusters:
$$\bar \xx = \frac{1}{N}\sum_{i=1}^Nx_i\qquad \bar\vv = \frac{1}{N}\sum_{i=1}^Nv_i,
\qquad \bar \xx_\A=\frac{1}{\# \A}\sum_{i\in \A}x_i,  \qquad \bar \vv_\A=\frac{1}{\# \A}\sum_{i\in \A}v_i.$$

\subsection{Main results}\label{sec:main}
Hereinafter, we fix our attention by referring to the individuals following the DI protocol as 
``particles'' with ``position'' $x_i$ and ``velocity'' $v_i$ (of $i$th particle).
Inter-particle interactions within the DI model highly depend on the local density of the ensemble,
and thus the following notion of densely packed clusters is of significant importance.

\begin{defi}\label{packed}
    We say that the cluster $\A\subset\{1,...,N\}$ is \underline{$r$-densely packed  at $t$} if
    \begin{enumerate}
        \item the Minkowski sum $\xx_\A(t-h)+B(0,r/2)$ is connected,
        \item each open ball $B(x_k(t-h),r)$, for $k\in \A$ contains more than $m$ particles,
    \end{enumerate} 
    with convention that $t-h= 0$ if $t\leq h$.
\end{defi}

\noindent
As we show later in Lemma \ref{con}, interactions between the particles in any $r$-densely packed 
cluster for $r\leq\delta$ are propagated to the entire cluster, which serves as a stepping 
stone to obtain a flocking estimate.

We perform basic qualitative and quantitative analysis of \eqref{eq1}. That includes:
the existence and uniqueness of classical solutions, conditional 
flocking estimates, and conditional cluster stability in terms of cluster density.
Moreover, we provide analytical examples signifying the rich dynamics of the system.
Finally, we illustrate the variability of possible behaviors by a number of numerical simulations.

The main results of the paper read as follows. 

\begin{theo}\label{main1}
Given the time interval $[0,T)\subset [0,+\infty)$, parameters $M,m,\delta,h>0$ and initial data $(\xx_0,\vv_0)\in \R^{2dN}$, the system \eqref{eq1} admits a unique classical $W^{1,\infty}([0,T))$ solution with 
nonincreasing velocity fluctuations
\begin{equation}\label{V}
V(t):=\max_{i,j\in\{1,...,N\}}|v_i(t)-v_j(t)|\leq V(s)\leq V(0),\qquad 0<s\leq t\leq T.
\end{equation}
\end{theo}

\noindent
In the following theorem $M$ is a positive function of $i$, $N$ and $\#{\mathcal N}_i$, thus
it has a finite set of possible values,  and thus it is bounded
below by a~constant $M_*>0$.  It is further discussed in Section \ref{sec:nM}. 

\begin{theo}[Dense clusters flock]\label{main2}
Given the time interval $[0,T)\subset [0,+\infty)$, parameters $M,m,\delta,h >0$ and initial data $(\xx_0,\vv_0)\in
\R^{2dN}$ suppose that the entire ensemble $\{1,...,N\}$ is $r$-densely packed with $r<\delta$. 
For sufficiently large $M$ the ensemble remains at least $\delta$-densely packed for all $t>0$ and the particles flock
exponentially fast. 

In other words there exists a constant $\lambda$, depending only on the dimension $d$ and on
$N$, such that
\begin{align*}
|v_i-\bar{\vv}|\leq e^{-M_*\lambda t} |v_i(0)-\bar\vv|,
\end{align*} 
provided that
\begin{align}\label{md}
     M_*:=\inf M >\frac{2}{\lambda(\delta-r)}.
\end{align}
\end{theo}

\begin{rem}\rm
Theorem \ref{main2} concerns the whole ensemble $\{1,...,N\}$ but also holds for any cluster ${\mathcal A}\subset \{1,...,N\}$ with the a priori assumption that ${\mathcal A}$ is sufficiently far away from the rest of the particles so that it is never influenced by them.
\end{rem}

\begin{rem}\rm
Assumption \eqref{md} is first and foremost an assumption on $M$; since the DI model has local interactions, similarly to the short-range CS model, flocking occurs with high coupling strength \cite{mpt}. However it can be viewed also as an assumption on the initial density governed by the constant $r$. With small $r$ the cluster is initially more densely packed and the right-hand side of \eqref{md} decreases.
\end{rem}



To illustrate possible behaviors of particles governed by the DI model,
we provide  the following analysis of clusters' interactions. 

\noindent
{\it {\bf Clusters' interactions.}
Let us consider two connected to each other clusters ${\mathcal A}$ and ${\mathcal C}$. Depending on the positions and velocities of individual particles within each cluster the following scenarios may occur: 
\begin{enumerate}[(\textnumero 1)]
    \item \underline{Stability}: Cluster ${\mathcal A}$ remains connected. Cluster ${\mathcal C}$ eventually detaches from ${\mathcal A}$ and they move separately.
    \item \underline{Breaking}: Cluster ${\mathcal A}$ breaks under the influence of cluster ${\mathcal C}$. Cluster ${\mathcal C}$ detaches from ${\mathcal A}$ together with a number of particles from ${\mathcal A}$.
    \item \underline{Sticking}: Cluster ${\mathcal A}$ is diverted by cluster ${\mathcal C}$. They remain connected indefinitely and their total momentum changes in the direction of the total momentum of ${\mathcal C}$.
\end{enumerate} 
}

The example showcasing all of the above scenarios, with a special choice of clusters ${\mathcal{A}}$ and 
$\mathcal{C}$, can be found in Section \ref{sec:exam}.

\subsection{Comparison with other models}\label{sec:comparison}

In recent years, it was recognized that the global all-to-all character of interactions in the Cucker-Smale (CS) flocking model not always corresponds to the actual behavior of agents in real-life phenomena, 
be it flocks of birds, networks of unmanned aerial vehicles or in opinion dynamics. The main issue is that, usually, the range of communication between autonomous agents is finite, the interactions are not symmetric and 
the structure of the network of interactions is not necessarily immersed in the standard Euclidean geometry.
For instance, the range of perception of a bird within a flock tends to be a finite cone-shaped area in front of it. Responding to these issues a number of non-standard alignment CS-type models emerged recently. Among them is the model with short-range interactions (CS$_\delta$ model), the $q$-closest neighbors model (CS$_q$ model) and the model with interference topology (CS$_t$ model). Below we summarize similarities and differences between these models and the DI model. Further, numerical comparison can be found in Section \ref{sec:sim}.

$\diamond$ {\it Short-range model CS$_\delta$}. The CS$_\delta$ model is a simple modification of the standard CS model in which the smooth and decreasing communication weight $\psi_\delta$ is assumed to be compactly supported on the set $[0,\delta]$. For instance, the classical CS weight $\psi(s) = (1+s)^{-\frac{1}{2}}$ could be cut-off by taking $\psi_\delta(s) = \psi(s) \chi_{[0,\delta]}$. With such a modiffication qualitative behavior and particularly asymptotics becomes significantly more difficult to study. Compared to the DI model, CS$_\delta$ model is symmetric and purely geometrical with $j\in \cN_i(t)$ if and only if $|x_i(t)-x_j(t)|\leq\delta$. For further information on the CS$_\delta$ model we refer to \cite{mpt,Jin,H-K-P-Z,S-T,S-D}.

$\diamond$ {\it The model with $q$-closest neighbors CS$_q$}. The CS$_q$ model is a modification of the CS model with the sum in \eqref{cs} taken over only those $j$ that are at most $q$-closest to $i$ in terms of position. This model is both density dependent and non-symmetric, with an opposite influence of the density to the DI model. While in the DI model high density is used to propagate the interactions, in the CS$_q$ model the interactions spread over constant mass of the particles and thus with low density the interactions reach further. For further information on the CS$_q$ model we refer to \cite{C-D2}.

$\diamond$ {\it Topological model with interference CS$_t$}. Recently, Shvydkoy and Tadmor introduced a CS-type alignment model, in which the interactions between the particles depend on the mass of the particles belonging to a symmetric area between them. The higher the mass of the particles between $x_i$ and $x_j$, the lower the interaction, which justifies the interpretation that intermediate particles interfere with the communication. The CS$_t$ model is density-dependent, symmetric and, similarly to CS$_q$, the influence of the density seems to be the opposite to the DI model. For further information on the CS$_t$ model we refer to \cite{S-T}.

The main difference between the DI model and the aforementioned is that the agents interact only with a densely-packed crowd (defined by the threshold $m$). Singletons, outsiders and agents forming low-density clusters do not interact. As a somewhat surprising consequence we obtain the possibility of a significant influence of the outsider particle on a cluster as showcased in Fig.~\ref{fig:2}. While, interactions of crowds with highly influential individuals were studied in the past \cite{Li,shen,C-D}, we obtain such phenomenon naturally within the framework of the model without artificially boosting the influence of any individual. Whether an individual manages to influence a cluster significantly depends not only on the models parameters but also on the spatial structure of the clusters, c.f. Section \ref{sec:52}.

Further state of the art for related models, including asymptotics \cite{car,H-K-P-Z,cons2,cuc1} and pattern formation \cite{top,ckpp}, collision avoidance \cite{ccmp,io} including models with singular interactions \cite{jpe,jps,Z-Z}, time delay \cite{C-H} and the recently emerging thermodynamical consistence \cite{choitermo,hatermo} can be found in the surveys \cite{B-D-T1} (CS model with regular interactions) and \cite{B-D-T2} (CS model with singular interactions). The surveys include also references to kinetic \cite{cons,can,mp} and hydrodynamic \cite{K-C,Tad1,Tad3,Shv,dmpw} limits of the CS model.

\subsection{On the normalizing factor $M$}\label{sec:nM}
There are many ways to normalize a consensus particle system, with the most popular prominently represented in the works by Cucker and Smale \cite{cuc1} and Motsch and Tadmor \cite{mo}. The first one, corresponding to $M = \frac{\kappa}{N}$, is the
 standard ''flat'' Cucker-Smale normalization, which connects the magnitude of the interactions to the total mass of the
  particles. The second one, $M=\frac{\kappa}{\#\cN_i}$, corresponds to the non-symmetric Motsch-Tadmor normalization.
   Throughout the paper we shall use a more general assumption that $M$ is any
   positive function of $N$, $i$ and $\#\cN_i$ i.e.
\begin{align}\label{norm}
M=M(N,i,\#\cN_i)\geq M_*>0.
\end{align}
The right-most inequality above follows from the fact that $M$ is a positive function on a discrete domain (with fixed $N$). Of course, with $M$ as in \eqref{norm} we incorporate both Cucker-Smale and Motsch-Tadmor normalizations.

The remainder of the paper is organized as follows. In Section \ref{sec:prelim} 
we reformulate the problem in terms of graphs, which, while interesting in itself, serves in the 
proof of Theorem \ref{main2}. Section \ref{sec:proof} and Section \ref{sec:exam} are dedicated to 
prove the main results: Theorems \ref{main1} and \ref{main2}, and to showcase the possible 
interaction between clusters (\textnumero 1 -- \textnumero 3). Finally, in Section~\ref{sec:sim}, we 
present numerical simulations.

\section{Preliminaries}\label{sec:prelim}

The analysis of asymptotic or otherwise qualitative behavior of system \eqref{eq1} is based on two foundations. 
The first is the analysis of how any particular configuration of particles $\xx$ influences their neighbor sets. This issue is closely related to the notion of $r$-densely packed clusters. 
The second foundation is focused on the translation of the propagation of interaction to the asymptotic behavior of the particles. 
This process was widely studied in the framework of graph theory (see for instance the works by Olfati-Saber \cite{Olfati1,Olfati2}). 
This motivates the following reformulation of \eqref{eq11}. Recalling the possible dependence of $M$ on $i$ in \eqref{norm} we take
\begin{equation}
\Phi_{ik} = \frac{M_i}{M_*}\quad \mbox{iff}\quad k\in\cN_i,\quad\mbox{otherwise}\quad\Phi_{ik} = 0,
\end{equation}
and we rewrite \eqref{eq1} as
\begin{align}\label{eq11}
\dot x_i=v_i,\qquad \dot v_i = M_*\sum_{k=1}^N\Phi_{ik}(v_k-v_i).
\end{align}

\noindent
Note that the lack of symmetry of the adjacency relation between the particles implies that usually $\Phi_{ik} \neq \Phi_{ki}$.

Equation \eqref{eq11} above, can be further restated as a model on graphs. To this end for any fixed time $t>0$ let $G=G(t)$ with
$$G=(V, {\mathcal E}, \PP) \mbox{ \  be a digraph (or directed graph)},
$$ 
where $V$ is the set of $N$ nodes corresponding to the $N$ particles and ${\mathcal E}={\mathcal E}(t)\subset V\times V$ is the set of edges between the nodes, that represents the connectivity of $G$. Matrix $\PP=\PP(t) = [\Phi_{ik}]$ describes interaction weights: $k$th node is interacting with $j$th node iff $\Phi_{ik} > 0$, or equivalently iff $k\in \cN_i$.
Then \eqref{eq1b} can be restated as
\begin{align}\label{eqm}
\dot \xx =\vv,\qquad \dot \vv = M_*(\DD-\PP)\vv,\qquad \DD = {\rm diag}(d_1,...,d_N),
\end{align}
where $d_i=\frac{\# \cN_i M_i}{M_*}$ is the ($M$-scaled) degree of the $i$th node of $G$. 

The importance of the reformulation \eqref{eqm} lies in its usefulness in the proof of Theorem \ref{main2} following a 4-step strategy as presented in the diagram:
$$
\begin{array}{ccc}
  \text{densely packed clusters}   \quad  & \longrightarrow & \quad \mbox{connectivity of}\ G \\ 
  \uparrow  &  &\downarrow  \\ 
 \mbox{propagation of the density} \quad  & \longleftarrow &  \mbox{control over velocity}
\end{array} 
$$

\noindent
It can be summarized as follows. Starting with a densely packed cluster at $t=t_0$ we ensure (see 
Lemma \ref{con} below) that interaction between particles in such a cluster is propagated to every particle. In other words, the graph $G(t_0)$ is undirected and connected. Using connectivity of $G(t_0)$, we apply theory developed in \cite{Olfati1,Olfati2} to establish local-in-time exponential decay of the relative velocity of the particles (see Lemma \ref{flock} below). This estimate enables us to control the changes in the positions of the particles and in turn -- it ensures the propagation of the density. This leads back to the densely packed clusters at some $t>t_0$ and an indefinite repetition of the scheme.

\noindent

\section{Proof of the main results}\label{sec:proof}
In this section we prove Theorem \ref{main1} and Theorem \ref{main2}.

\subsection{Proof of Theorem \ref{main1}}

\begin{proof}{(\it Theorem \ref{main1}: Existence.)}  

By the definition of $\cN_i$ the system is elementarily solved on $[0,h]$ -- it is a linear system with constant coefficients. Its unique solution is a linear combination of functions of the type $t^k e^{\alpha t}$, with $k \leq N$ and $\alpha$ determined by eigenvalues of the matrix appearing in (\ref{eqm}).

The main difficulty is that neighborhoods $\cN_i$ change with time, which causes  discontinuities of the right-hand side in \eqref{eqm}. This is not particularly problematic unless the number of discontinuities is infinite. Thus, we want to exclude an oscillatory behavior of solutions.

We proceed by induction. Suppose that the solution exists on $[(k-1)h,kh]$ for some $k=1,2,...$ and the sets of neighbors $\cN_i$ change finitely many times for each particle in the time interval $[(k-1)h,kh]$. Then in the time interval $[kh, (k+1)h]$ the system is well defined with a finite number of discontinuities. In particular $[kh, (k+1)h]$ can be decomposed into a finite collection of intervals on which the system is linear. On each of these intervals a unique solution exists and is a linear combination of functions of the type $t^k e^{\alpha t}$. These intervals can be glued together and while the smoothness of the solution may be lost at the endpoints, the
right-hand side of \eqref{eq1} is a bounded function of $\xx,\vv$ and thus the solution belongs at least to $W^{1,\infty}([kh,(k+1)h])$ and is continuous and unique. Then such a solution, as a piecewise linear combination of analytic functions, cannot exhibit an oscillatory behavior and the number of times of discontinuity in $[kh, (k+1)h]$ is finite. By induction, this procedure can be repeated indefinitely and the existence of solutions with $\xx$ and $\vv$ continuous and $\vv\in W^{1,\infty}([0,T))$ is proved for any $T>0$.

\end{proof}

\begin{rem}\rm\label{Rh}
The reason to introduce the parameter $h>0$ is to circumvent the problems arising in the following scenario. Suppose that $h=0$ and at $t=0$ we have $|x_1(0)-x_2(0)| = \delta$ and $v_1(0)=v_2(0)$. Then it is unclear whether at $t>0$ the particles will be within each others range. Introducing $h>0$ is an easy way to ignore this problem as shown in the above proof.  
\end{rem}

In order to prove the decay of velocity fluctuations we apply the $L^\infty$ method used, for instance in \cite{mo}.
\begin{proof}{(\it Theorem \ref{main1}:  Decay of velocity fluctuations.)}   
 Fix $t>0$ and let
 the maximal relative velocity between the particles be realized by $i$th and $j$th particles 
 i.e.
    \begin{align*}
    V(t) = |v_i(t)-v_j(t)|.
    \end{align*}
 Using the \eqref{eq11} formulation of the velocity equation, we have
    \begin{align}\label{eq2}
    \frac{1}{2}\frac{\dd}{\dd t}(v_i-v_j)^2 = (v_i-v_j)\cdot\Bigg(M_*\sum_{k=1}^N\Phi_{ik}(v_k-v_i) - M_*\sum_{k=1}^N\Phi_{jk}(v_k-v_j)\Bigg).
    \end{align}
    Observe that $\Phi_{ii}$ does not play any role in \eqref{eq2}, since it is multiplied by $v_i-v_i = 0$. Thus, denoting $M^*:=\sup M$, we redefine it for all $i\in\{1,...,N\}$ as 
    $$
    \Phi_{ii}:= \frac{N M^*}{M_*} - \sum_{k\neq i}\Phi_{ik}> 0,
    $$
    so that 
    $$
    M_*\sum_{k=1}^N\Phi_{ik} = M^* N \mbox{ \ \ for each  \ } i\in\{1,...,N\}.
    $$
    Then, we rewrite \eqref{eq2} as 
    \begin{equation}\label{eq3}
    \begin{split}
    \frac{1}{2}\frac{\dd}{\dd t}(v_i-v_j)^2&= (v_i-v_j)\cdot M^*N(v_j-v_i) + (v_i-v_j)\cdot M_*\sum_{k=1}^N(\Phi_{ik}v_k - \Phi_{jk}v_k)\\
&= -M^* NV^2 + N(M^*-M_*\eta)(v_i-v_j)\cdot\sum_{k=1}^N(\alpha_{ik}v_k - \alpha_{jk}v_k),
    \end{split}
    \end{equation}
    where
    \begin{align*}
    \alpha_{ik} &:= \frac{M_*(\Phi_{ik}-\eta)}{N(M^*-M_*\eta)},\qquad \eta:=\frac{1}{2}\min_{k,l\in\{1,...,N\}}\Phi_{kl}\in [0,M^*/2 M_*],\\
    \alpha_{ik} &\geq 0,\qquad \sum_{k=1}^N\alpha_{ik} = 1.
    \end{align*}
    Therefore the right-most term in \eqref{eq3} is a convex combination of elements in ${\rm conv}\{v_k: k\in\{1,...,N\}\} + {\rm conv}\{-v_k: k\in\{1,...,N\}\}$ and thus -- it can be bounded from the above by the maximal diameter of such a set, which is $V$. Consequently, we have
    \begin{align*}
    \frac{1}{2}\frac{\dd}{\dd t}V^2 \leq -M_* N\eta V^2.
    \end{align*}
    In the worst case scenario $\eta=0$, so we deduce that $V$ does not increase and the proof is finished.
    
\end{proof}

\begin{rem}\rm
The $L^\infty$ method applied in the above proof is rather complicated for what it can accomplish for the DI model. However, we chose it to showcase why there is no unconditional exponential flocking, which is precisely because the minimal decay rate $\eta$ can be $0$. On the other hand if we
suppose that all particles are at most at distance $\delta$ from one another and the number of particles is greater than $m$, $\Phi_{ik} = \frac{M_i}{M_*}$ for all $i,k\in\{1,...,N\}$,  then in the above proof $\eta = 1/2$ and $V$ decays exponentially with exponent $\lambda =  M_* N/2$. This is however of no surprise since then equation \eqref{eq1b} is reduced to
\begin{align*}
\dot v_i = \sum_{k=1}^N M_i(v_k-v_i),
\end{align*}
which is a simple linear model of consensus. The difficulty lies in the propagation of the
interaction between the particles  if only some of the $\Phi_{ik}$ are positive. 
It is addressed in Theorem \ref{main2} and particularly in Lemma \ref{flock}, presented below.
\end{rem}

\subsection{Proof of Theorem \ref{main2}}

To prove Theorem \ref{main2} we utilize the 4-step strategy described in Section~\ref{sec:prelim}. 
First we show that any $r$-densely packed cluster with $r\leq\delta$ 
induces an undirected and strongly connected subgraph of $G$. That is to say, the interaction 
propagates throughout any $r$-densely packed cluster of the particles.

\begin{lem}\label{con}
    If $\A$ is an $r$-densely packed cluster at $t\geq 0$ for $r\leq \delta$ then for all $i,j\in \A$ we have
\begin{enumerate}[(i)]
\item $i\in \cN_j$ iff $j\in \cN_i$,
\item there exists a sequence $\{l_k\}_{k=1}^p$ such that 
    \begin{align*}
    l_1 = i,\quad l_p = j\quad\mbox{and}\quad l_k\in \cN_{k+1}\quad\mbox{for all}\quad k\in\{1,...,p-1\}.
    \end{align*} 
\end{enumerate}    
In other words $\A$ treated as a subgraph of $G(t)$ is undirected and (strongly) connected.
\end{lem}

\begin{proof}
   Fix $t\geq 0$. Since, for $r<\delta$ an $r$-densely packed cluster is also $\delta$-densely packed, 
    we will prove the lemma assuming that $r=\delta$.
    
    First observe that condition 2 in Definition \ref{packed} implies that for all pairs $i,j\in \A$ we have $i\in \cN_j$ if and only if $|x_i(t-h)-x_j(t-h)|<\delta$. 
    Thus $G(t)$ is undirected and it remains to show that it is connected (then it is automatically also strongly connected). Now suppose that the lemma is not true and there exists a $\delta$-densely
    packed cluster $\A$ that is not connected. Then there exist two particles with indexes $i$ and $j$ without a connecting sequence $\{l_k\}_{k=1}^p$.
    Let $\A_i$ and $\A_j$ be the maximal connected clusters including $i$ and $j$, respectively. Then ${\rm dist}(\A_i, \A_j)> \delta$, since otherwise we could 
    find $q_i\in \A_i$ and $q_j\in \A_j$ such that $|x_{q_i}(t-h)-x_{q_j}(t-h)|\leq\delta$ and $\A_i\cup \A_j$ would be connected through the pair $(q_i,q_j)$. However with 
    ${\rm dist}(\A_i, \A_j)> \delta$ condition 1 in Definition \ref{packed} is not satisfied. Contradiction with the assumption that $\A$ is not strongly connected finishes the proof.
\end{proof}

The next lemma uses the propagation of interaction within densely packed clusters, ensured by Lemma \ref{con}, to provide a local flocking estimate. 

\begin{lem}\label{flock}
    Let $\A$ be a cluster that is strongly connected and undirected in the interval $[t_1,t_2)$. Assume further that no particle outside of $\A$ influences $\A$. Then the cluster average velocity $\vv_\A=\frac{1}{\# \A}\sum_{i\in \A}v_i$ is constant in $[t_1,t_2)$ and there exists $\lambda>0$ such that
    \begin{align*}
    |v_i(t)-\bar \vv_\A|\leq e^{-t M_*\lambda}|v_i(t_1)- \bar \vv_\A|\qquad \mbox{for all}\quad i\in \A \quad \mbox{and all}\quad t\in[t_1,t_2),
    \end{align*}
    where $M_*$ is the minimal value of the function $M$, c.f. \eqref{norm}.
\end{lem}

\begin{proof}
    If $\A$ is a connected undirected graph with no outside influence then it can be treated as a separate ensemble of particles following equation \eqref{eqm} with symmetric $\PP$. Then $\frac{\dd}{\dd t} \vv_\A = 0$ and $\bar \vv_\A = \bar\vv_\A(t_1)$. Furthermore it is well known, see for instance \cite{Olfati1,Olfati2}, that each $v_i$ converges exponentially fast to $\bar \vv_\A$ with exponent equal to the (positive) second eigenvalue $\lambda_2(t)$ of the matrix $\LL(t)=\DD(t)-\PP(t)$ (multiplied in our case by $M_*$). Since matrices $\DD(t)$ and $\PP(t)$ have only finite number of possible distinct entries (for a given $M$ and $N$), then there is also finitely many possible distinct matrices $\LL(t)$, which implies that   
    \begin{align*}
    \inf_{t\in[t_1,t_2)} \{\lambda_2(t)\} = \min_{t\in[t_1,t_2)}\{\lambda_2(t)\} =: \lambda>0.
    \end{align*}

    Then all velocities converge to $\bar\vv_\A$ exponentially fast with exponent $\lambda M_*$ (in the time interval $(t_1,t_2)$). 
\end{proof}

Finally we prove the following slightly reformulated version of Theorem \ref{main2}.

\begin{prop}\label{main21}
    Fix $(\xx_0,\vv_0)$. Let the initial ensemble of the particles $\{1,...,N\}$ associated with positions $\xx_0$ and velocities $\vv_0$ be $r$-densely packed with $r<\delta$. Then the ensemble remains densely packed for all $t>0$ and the particles flock exponentially fast, provided that
\begin{align*}
    M_* >\frac{2}{\lambda(\delta-r)}.
    \end{align*}    
\end{prop}

\begin{proof}{(\it Proposition \ref{main21} and Theorem \ref{main2}.)}
    Since, initially, the particles are $r$-densely packed for $r<\delta$ and the maximal velocity of the particles is uniformly bounded, see Theorem \ref{main1}. Thus, there exists a time interval $[0,T]$ such that the particles are $\delta$-densely packed for $t\in[0,T]$. Then by Lemma \ref{con} and Lemma \ref{flock}, for $t\in[0,T]$, the ensemble is a connected undirected cluster, the average velocity of the particles is constant and each velocity $v_i$ converges to the average exponentially fast with exponent $M_*\lambda$. Let
    \begin{align*}
    T^* := \sup\{T>0: \{1,...,N\}\mbox{ is } \delta\mbox{-densely packed for all }t\leq T\}>0.
    \end{align*}
    
    We shall prove that if $M_*$ is large enough then actually $T^*=\infty$. Assuming the contrary, for all $i\in\{1,...,N\}$ and all $t\in[0,T^*]$, denoting
    \begin{align*}
    x_i(t)-\bar\xx(t) &= \int_0^tv_i(s)\dd s + x_i(0) - t\bar\vv - \bar\xx(0)\\
    &= \int_0^t(v_i(s) - \bar\vv)\dd s + (x_i(0)-\bar\xx(0)).
    \end{align*}
    which for any $j\in\{1,...,N\}$ leads to
    \begin{equation}\label{cin}
    \begin{split}
    |x_i(t)-x_j(t)| &= |(x_i(t)-\bar\xx(t)) - (x_j(t) - \bar\xx(t))|\\
    & = \left|\int_0^t(v_i(s)-\bar\vv)\dd s + \int_0^t(\bar\vv-v_j(s))\dd s + (x_i(0)-x_j(0))\right|\\
    &\leq \int_0^t|v_i(s)-\bar\vv|\dd s + \int_0^t|v_j(s)-\bar\vv|\dd s + |x_i(0)-x_j(0)|.
    \end{split}
    \end{equation}
Let us fix in \eqref{cin} any pair $(i,j)$ such that $|x_i(0)-x_j(0)|\leq r$. Such pairs exist since the ensemble is $r$-densely packed. On $[0,T^*)$ the ensemble is densely packed and thus, by Lemmas \ref{con} and \ref{flock} we have
    \begin{align*}
    |x_i(t)-x_j(t)|\leq 2\int_0^\infty e^{-M_*\lambda t}|v_i(0)-\vv| + r = \frac{2}{M_*\lambda} + r.
    \end{align*}
    Therefore if
    \begin{align*}
    M_* > \frac{2}{\lambda(\delta-r)},
    \end{align*}
    then $|x_i(t)-x_j(t)|<r_*\leq\delta$ for all $t\in[0,T^*]$. This implies that any two particles of distance not greater than $r$ initially are of distance at most $r_*$ from one another. It is therefore easy to see that for all $t\in[0,T^*]$ the $r_*$-neighborhood of the ensemble is connected and the number of particles in each $B(x_i,r_*)$ is larger than $m$. Thus the ensemble is $r_*$-densely packed in  $[0,T^*]$. In particular, the ensemble is $r_*$-densely packed at $T^*$ and, again by boundedness of the velocity, we can find $\epsilon>0$ such that the ensemble is $\delta$-densely packed in $[0,T^*+\epsilon)$, which contradicts the assumption that $T^*<+\infty$. Consequently $T^*=\infty$ and the ensemble is $\delta$-densely packed, strongly and symmetrically connected and it flocks exponentially fast to the average velocity, which is constant. The proof is finished.
\end{proof}

\section{Clusters' interaction}\label{sec:exam}

In what follows we perform the analysis of clusters' interactions as introduced at the end of Section \ref{sec:main}. For this sake we consider a simple yet significant
setting of cluster interactions, where one cluster (${\mathcal C}$) is a single particle ($c$). 
In such a case the three scenarios  read:
\begin{enumerate}[\textnumero 1]
\item \underline{Stability}: Cluster $\A$ remains connected. Particle $c$ detaches from $\A$.

\smallskip 

\item \underline{Breaking}: A single particle $b$ from $\A$ detaches from $\A$ under the influence of~$c$.

\smallskip 

\item \underline{Sticking}: Cluster $\A$ changes its momentum under the influence of $c$.
\end{enumerate}

\begin{figure}[h]
    \begin{center}
        \setlength{\unitlength}{.1\textwidth}
        \begin{picture}(10,.1)
        \put(2,0){\line(1,0){6}} 
        \put(3,-.05){\line(0,1){.1}} 
        \put(2.95,-.3){0} 
        \put(2.95,.2){$a$} 
        \put(2.9,0){\circle{0.1}} 
        \put(2.95,0){\circle{0.1}} 
        \put(3,0){\circle{0.1}} 
        \put(3.05,0){\circle{0.1}} 
        \put(3.1,0){\circle{0.1}} 
        \put(4.,0){\circle{0.1}} 
        \put(4,-.05){\line(0,1){.1}} 
        \put(3.95,-.3){$\beta$} 
        \put(3.95,.2){$b$} 
        \put(5,-.05){\line(0,1){.1}} 
        \put(4.95,-.3){$\delta$} 
        \put(5.5,0){\circle{0.1}} 
        \put(5.5,-.05){\line(0,1){.1}} 
        \put(5.45,-.3){$\gamma$} 
        \put(5.45,.2){$c$} 
        \end{picture}
    \end{center}
    
    \caption{Scheme of the considered Cluster setting.}
    \label{fig:1dsetting}
\end{figure}
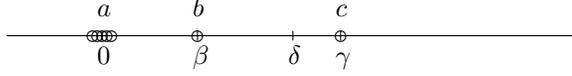

{\it The setting.}

Suppose that we have $N+1>m+1$ particles in $\R^2$ distributed as in Fig.~\ref{fig:1dsetting}.
All of the particles are initially situated on the horizontal $x_1$-axis. 
The left-most $N$~particles form the cluster $\A$. 
Within this cluster the leftmost $N-1$ particles are initially in the position near zero with zero velocities.
The distances between them are significantly smaller then $\delta$. 
We shall refer to all of these particles as $a$ denoting its approximate position as $x_a(0)= 0$.
The rightmost particle ($b$) in cluster $\A$ 
is in the position $x_b(0)=(\beta,0)$ with zero initial velocity and $\beta < \delta$. 
The last particle under considerations, the singleton ($c$) is positioned at $x_c(0) = (\gamma,0)$ with 
$\gamma > \delta$ and  $\gamma - \beta < \delta$; the singleton's initial velocity is $(0,v_c)$.
Since we assume that $0<h\ll\beta,\delta,\gamma,r,1/N$, for simplicity of the presentation we set $h=0$.

Note that not only is cluster $\A$ $r$-densely packed with $r\leq \delta$ but actually each particle in $\A$
interacts with each other particle.  Moreover, the only non-zero initial
velocity of the entire ensemble is $(0,v_c)$, which implies that any change in the velocity of any one of the
particles can only occur in direction parallel to $(0,v_c)$.

To summarize, we have the following picture:
\begin{equation}\label{id}
\begin{split}
 x_a(0)\approx (0,0), \;\; v_a(0)=(0,0), &\quad \mbox{it describes the set of $N-1$ particles}, \\
 x_b(0)=(\beta,0), \;\; v_b(0)=(0,0), & \quad 0<\beta <\delta,\\
 x_c(0)=(\gamma,0),\;\; v_c(0)=(0,c), & \quad \gamma > \delta \mbox{ \ and \ } \gamma - \beta < \delta,
\end{split}
\end{equation}
with $x_a$ serving as a stand-in for all $x_i$ with $i$ in the left-most cluster $a$. Then, for $t>0$ we let the
particles follow the protocol \eqref{eq1}. In what follows we successively reformulate and simplify \eqref{eq1}
in the configuration \eqref{id} leading to a formulation that can be expressed in terms of $v_c$. First note that
configuration \eqref{id} leads to a~characterization of the neighbor sets:
\begin{align*}
\cN_a(0) = A = a\cup\{b\},\quad \cN_b(0) = A\cup\{c\},\quad \cN_c = \emptyset, 
\end{align*}
which implies that \eqref{eq1} takes the form:

\begin{equation}\label{sys0}
\begin{split}
\dot v_a &= M(v_b-v_a),\\
\dot v_b &=M(N-1)(v_a-v_b) + M(v_c-v_b),\\
\dot v_c &=0.
\end{split}
\end{equation}

Since the motion can occur only in directions parallel to $(0,v_c)$, we assume, without a loss of generality,
that the velocities appearing in \eqref{sys0} are scalar. Moreover, through a simple scaling argument
(e.g. $\frac{\dd}{\dd t}\bar{v}_a(t) = \frac{\dd}{\dd t} v_a(Mt)$ etc.) we may further simplify assuming
that $M=1$. Then \eqref{sys0} reduces to a scalar ODE
\begin{equation}\label{sys1}
 \begin{split}
  \dot v_a &= -(v_a-v_b),\\
  \dot v_b &= -(N-1)(v_b-v_a) - (v_b-v_c),
 \end{split}
\end{equation}
with initial data $v_a(0)=0,v_b(0)=0$ and a constant $v_c$.
Then
\begin{equation*}
 \frac{\dd}{\dd t}(v_b-v_a) =-N(v_b-v_a) - (v_b-v_c),
\end{equation*}
which, with constant $v_c$, implies
\begin{equation}\label{x3}
 (v_b-v_a)(t)=-e^{-Nt}\int_0^t e^{Ns}(v_b(s)-v_c)\dd s.
\end{equation}
Inserting (\ref{x3}) to (\ref{sys1}) yields
\begin{equation*}
 \dot v_b = - (N-1)\left(-e^{Nt}\int_0^t e^{Ns}(v_b(s)-v_c)\dd s\right) - (v_b-v_c).
\end{equation*}
Thus $v_b$ satisfies the second order ODE
\begin{equation*}
 \ddot v_b + (N+1)\dot v_b + v_b-v_c =0
\end{equation*}
with initial data $v_b(0)=0$ and $\dot v_b(0)=v_c$. Since  $v_c$ is a constant, introducing 
\begin{equation*}
 V_b=v_b-v_c \mbox{ with } V_b(0)=-v_c\text{ and } \dot V_b(0)=v_c,
\end{equation*}
we obtain
\begin{equation*}
 \ddot V_b+(N+1)\dot V_b + V_b=0.
\end{equation*}
By solving the above ODE, we find that
\begin{equation*}
 V_b(t)=V_1e^{-\lambda_1 t} + V_2e^{-\lambda_2t},
\end{equation*}
where
\begin{align*}
 \lambda_1 &= \frac{-(N+1) - \sqrt{ (N+1)^2 -4} }{2} \approx -(N+1),\\
 \lambda_2 &= \frac{-(N+1) + \sqrt{ (N+1)^2 - 4}}{2} \approx -(N+1)^{-1}.
\end{align*}
Here $\approx$ is the {\it asymptotic equality} as $N\to\infty$.

We recover $V_1$ and $V_2$ from the initial data by taking
 $V_1+V_2=-v_c \mbox{ \ and \ } -\lambda_1 V_1 -\lambda_2 V_2 = v_c.$
Therefore
\begin{equation*}
 V_1=v_c\frac{1-\lambda_2}{\lambda_2-\lambda_1} \approx v_c\frac1N \mbox{ \ \  and \ \ } V_2=v_c \frac{\lambda_1 -1}{\lambda_2-\lambda_1}\approx -v_c(1+\frac1N).
\end{equation*}

Hence we obtain the explicit formula on approximate $V_b(t)$:
\begin{equation}\label{conf}
\boxed{ V_b(t)\approx v_c \frac1N e^{-(N+1)t} - v_c\left(1+\frac 1N\right) e^{-(N+1)^{-1}t}\approx - v_c\left(1+\frac 1N\right) e^{-(N+1)^{-1}t}.}
\end{equation}

\bigskip
{\it The three scenarios.}

Our next goal is to obtain three distinct aforementioned scenarios within the configuration expressed by
\eqref{conf}. First we aim to distinguish scenarios \textnumero 1 and \textnumero 2 (the {\it stability} 
and {\it breaking} scenarios), with particle $c$ either failing to drastically change the structure 
of the cluster $\A$ or scooping particle $b$ out of the cluster $\A$. 
Assume that $N$ is large. By \eqref{conf},
For large $N$, \eqref{conf} gives
\begin{align}\label{v}
v_b-v_c = V_b \approx -v_c e^{-(N+1)^{-1}t},
\end{align}
thus for any $T>0$ we have
\begin{equation*}
\begin{split}
|x_b(T)-x_c(T)| = & |\gamma -\beta| + \left|\int_0^T (v_b(t)-v_c) \dd t\right| \\ &\approx |\gamma-\beta| +\left|\int_0^T -v_c e^{-(N+1)^{-1}s} \dd s\right| \approx |\gamma-\beta| +|v_cT|
\end{split}
\end{equation*}
and
\begin{equation*}
|x_a(T)-x_b(T)|\approx |\beta| + \left|v_c \int_0^T e^{-Nt}\int_0^t e^{Ns} e^{-(N+1)^{-1}s} \dd s\right|\approx |\beta| +\left|\frac{v_cT}{N}\right|.
\end{equation*}
The above equations imply that if $t=T$ is the breaking point for $c\in \cN_b(t)$ i.e.
\begin{align*}
T=\sup\big\{t: c\in \cN_b(s)\quad \mbox{for all}\quad s\in[0,t)\big\}
\end{align*}
or the breaking point for $b\in \cN_a(t)$ it needs to satisfy
\begin{equation}\label{break}
 |\gamma-\beta| +|v_c T|=\delta \mbox{ \ or \ } |\beta| +\left|\frac{v_c T}{N}\right|=\delta,
\end{equation}
respectively. Then scenarios \textnumero 1 and \textnumero 2 are distinguished as follows:

\bigskip

\begin{enumerate}[(\textnumero 1)]
\item With $\beta = \delta/2$ and $\gamma = \delta$ so that $\gamma-\beta = \delta/2$ in \eqref{break} we have
\begin{align*}
\delta = |\gamma-\beta| +|v_c T| = \frac{\delta}{2} + |v_c T| > \frac{\delta}{2} + \left|\frac{v_c T}{N}\right| = |\beta| +\left|\frac{v_c T}{N}\right|.
\end{align*}
The above implies that at $T$ the $c$ particle breaks out from $\cN_b$ while cluster $\A$ itself remains unaffected (in particular $b\in \cN_a$ and $a\in \cN_b$), provided that $N$ is large. Thus we obtain the {\it stability} scenario \textnumero 1.

\medskip 

\item With $\beta = \delta-\epsilon$ and $\gamma = \delta$ so that $\gamma-\beta = \epsilon$ with $0<\epsilon\ll\delta$ we have
\begin{align*}
\delta-\epsilon+\left|\frac{v_c T}{N}\right|= |\beta| +\left|\frac{v_c T}{N}\right|=\delta,
\end{align*}
if and only if
\begin{align}\label{ep2}
\epsilon=\left|\frac{v_c T}{N}\right|.
\end{align}
On the other hand, taking the above into the account yields
\begin{align*}
|\beta| +\left|\frac{v_c T}{N}\right| = \delta-\epsilon + \epsilon > \epsilon + |v_c T| = |\gamma-\beta| +|v_c T|
\end{align*}
which holds if $|v_c|$ satisfies
\begin{align}\label{ep1}
T|v_c|\left(\frac{1}{N}+1\right) < \delta.
\end{align}
If the above conditions are satisfied, the breaking point for $b\in \cN_a$ comes sooner than the breaking point for $c\in\cN_b$ and thus the $b$ particle is removed from the cluster $\A$, leading to the {\it breaking} \textnumero 2 scenario.

\medskip 

\item
The last, {\it sticking} \textnumero 3 scenario requires more preparation. For a fixed $N$ and small $v_c$, we recall \eqref{v} to see that
\begin{equation*}
V_b(t) = v_b(t)-v_c \approx -v_c e^{-\frac1N t}
\end{equation*}
And recalling from \eqref{x3} that
$$
v_b(t)-v_a(t) = -e^{-Nt} \int_0^t e^{Ns}V_b(s) ds \approx v_c\frac1N e^{-\frac1N t}.
$$
Integration of the above two equations leads to the uniform-in-time upper-bounds on the relative distance between the sub-clusters
$$
\sup_{t\geq 0}|x_c(t)-x_b(t)| \leq |\gamma - \beta| + \left|\int_0^\infty V_b(t) \dd t\right| \approx |\gamma -\beta| + |v_cN| 
$$
and
$$
\sup_{t\geq 0}|x_a(t)-x_b(t)| \leq |\beta|+ \left| \int_0^t (v_b-v_a)(s)\dd s\right| \approx |\beta| + |v_c|.
$$

Hence, preservation of the cluster's structure is ensured by taking
\begin{equation}\label{ep3}
 |\gamma -\beta| + |v_cN|  \leq \delta \mbox{ \  and \ }  |\beta| + |v_c|\leq \delta,
\end{equation}

which holds for small $v_c$. Such conditions ensure that the distance between the leftmost sub-cluster $a$ and the middle particle $b$ remains smaller than $\delta$ indefinitely; in other words cluster $\A$ is preserved for all $t>0$. Furthermore the distance between the middle particle $b$ and the rightmost singleton $c$ is also smaller than $\delta$ for all times. Hence $\A\cup\{c\}$ is a connected cluster. The connectivity of $\A\cup\{c\}$ is weak in the sense that while $\A$ is strongly and symmetrically connected (it is in fact $\delta$-densely packed) and $c$ influences $\A$ through the middle particle $b$, there is no influence directed from $\A$ to $c$ ($\cN_c = \emptyset$).

\end{enumerate}

\begin{rem}\rm
Let us describe how the total momentum (equal to the sum of velocities, since $M=1$) of the $N+1$ particles evolves in each case \textnumero 1 - \textnumero 3.
By the above considerations one has:
\begin{equation}
 v_b\approx v_c(1-e^{-T/N}) \mbox{ and } v_a\approx v_c(1-e^{T/N} - \frac{1}{N}e^{-T/N})
\end{equation}
In the first scenario the time of separation is of order 
$T \approx \frac{\delta}{v_c}$. Thus the momentum of the cluster increases (here $N+1\approx N$):
\begin{multline}
Mom:= v_b +N v_a = v_c (1-e^{-T/N}) + N v_c(1-e^{-T/N} -\frac{1}{N}e^{-T/N})\approx \\
v_c  \frac{T}{N} + Nv_c \frac{T}{N} + v_c e^{-T/N} \approx  \delta + {v_c} e^{-\delta/ (Nv_c)}
\end{multline}
In the second case by (\ref{ep2}), (\ref{ep1}) $T\approx \frac{\epsilon N}{v_c}$, but $N\epsilon <\delta$ and we get almost the same ($T \approx \frac{\delta}{v_c}$)
\begin{equation}
 Mom:= \delta + v_c e^{-\delta/ (Nv_c)}.
\end{equation}
In the last case, the time of separation $T$ is infinity, all velocities are reaching $v_c$, so by  condition  (\ref{ep3})
$ Mom:=Nv_c \approx \delta.$

Interestingly enough, in each of the scenarios \textnumero 1 - \textnumero 3 the growth of the momentum is almost the same.

\end{rem}

\section{Simulations}\label{sec:sim}

In this section we present the results of numerical simulation of the 
DI model in several different cases. The aim is twofold. First we 
compare the model with related models. Second we illustrate its unique features, 
and provide insight to the theoretical results of this work.



Let us recall that the DI model is characterized by the set of parameters $(N, M, m, \delta)$. 
Together with the volume of the domain, which in the case of two-dimensional square equals $L^2$, we 
can transform $N, m, \delta$ and $L$ to the average particle density $\rho_a$ and the local minimal 
interaction density $\rho_m$, namely 
$$\rho_a = \frac{N}{\bar L^2},\quad \rho_m = \frac{m}{\pi \delta^2} .$$

For the fixed $L  = 25$ in the case of cluster formations (Sect. \ref{sec:case1}) we choose 
$(N,m,\delta)=(64,3,2)$. In this example $\rho_m = 2.3\rho_a$  , thus the necessary condition for a 
particle interaction is to locally reach particle density that is more than two times of the 
average, however in practice since the values are discrete its in fact three times more.

Computational domain is a two-dimensional square with the periodic boundary condition. Periodicity 
is implemented by introducing domain extension of the size $\delta$ and ghost particle technique.

For the sake of visualisation we construct a directed graph that represents interactions $G=(V, {\mathcal E},\PP)$,
see Section~\ref{sec:prelim}. In each time-step we identify clusters as a strongly connected components 
of graph $G$, and prescribe them a color from predefined list of colors. 
To this end we apply an algorithm \emph{connected\_components} implemented  in \emph{scipy.sparse.csgraph} \cite{scipy}.

\paragraph{Numerical scheme\\}

As a numerical method we choose the classic 4th-order Runge-Kutta method (RK4), that is one of the most widely used method to solve ODEs. The application of RK4 algorithm to system \eqref{eq1}, where we renamed the rhs in \eqref{eq1b} as $a$, yield the following numerical scheme

\begin{subnumcases}{\label{rk}}
v^{n+1}_i=v^n_i + \frac{1}{6}\left(k_{v1}+ 2k_{v2} + 2k_{v3}+ k_{v4}\right),\label{rkv}\\
x^{n+1}_i=x^n_i + \frac{1}{6}\left(k_{x1}+ 2k_{x2} + 2k_{x3}+ k_{x4}\right),\label{rkx}
\end{subnumcases}
where 
\begin{align*}   
  k_{v1} &= a(x_i^n,v_i^n) \dt ,\quad & & k_{x1} = v_i^n \dt,   \\
  k_{v2} &= a(x_i^n+ k_{x1}/2 ,v_i^n+ k_{v1}/2) \dt ,\quad & & k_{x2} = \left(v_i^n + k_{v1}/2\right) \dt,   \\
  k_{v3} &= a(x_i^n+ k_{x2}/2 ,v_i^n+ k_{v2}/2) \dt ,\quad & & k_{x3} = \left(v_i^n + k_{v2}/2\right) \dt,   \\
  k_{v4} &= a(x_i^n+ k_{x3} ,v_i^n+ k_{v3}) \dt ,\quad & & k_{x4} = \left(v_i^n + k_{v3}\right) \dt   .
\end{align*}

\subsection{Cluster formations}\label{sec:case1}

In the first scenario we show spontaneous behaviour of particles in time. Starting from $N$ particles that are randomly distributed and initially placed with a certain distance from the domain boundary. We observe particle interactions and clusters formations based on local particle density.

Moreover, we present comparison with the Cucker-Smale (CS) type models. Beyond the classical CS 
model, two local versions of CS model are considered: CS$_\delta$ where any particle interacts only 
within the ball of radius $\delta$  and CS$_q$ where any particle interacts with $q-$closest 
neighbours. The weight function for CS models is given by $\psi(s) = 1/\sqrt{1+|s|}$.

The initial velocity is equal $v_i(0) = r_i(\cos(\alpha_i), \sin(\alpha_i)) + f_i , \forall i = 
\{1..N\}$ with $r_i$ and $\alpha_i$ being random numbers taken from the uniform distributions 
$U([0,1])$ and $U([0,2\pi])$ respectively. Moreover, in order to avoid the average velocity being 
zero, for half of the particles $i\le N/2$ we introduce the fixed contribution $f_i$ that is equal 
to $f_i = r_i(0.5,1)$.

For the sake of the comparison of different models we put $M$ in such a way that it is always equal 
to one over the number of particles within a cluster, c.f. Section~\ref{sec:nM}.
Tab.~\ref{tab:case1} summarizes considered models with corresponding set of parameters.

\begin{table}[h!]
    \centering
    \begin{tabular}{clccccc}
        \toprule
        &  \text{model} & $N$ & $M$ & $m$ & $\delta$ & $q$   \\	\midrule 
        (A) & DI   & 64 & $1/\# \cN_i$ & 3 & 2  & \\
        (B) & CS$_{\delta}$   & 64 & $1/\# \cN_i$ &  & 2 &  \\
        (C) & CS$_{q}$   & 64 & $1/q$ &  & $\infty$ & 3 \\
        (D) & CS   & 64 & $1/N$ &  & $\infty$ & \\
        \bottomrule 
    \end{tabular} 
    \caption{Summary of models considered in Section \ref{sec:case1}.}
    \label{tab:case1}
\end{table}

Fig.~\ref{fig:case1} presents the evolution of particles governed by DI, CS$_\delta$, CS$_q$ and 
CS models at $t=0$, $t=30$ and $t=150$. We prescribe the same colour for particles that interact 
with each other at least indirectly. In all of the cases we observe clustering and velocity 
alignment.

For DI and CS$_\delta$ we have the same range of interaction $\delta=2$, but DI model on the top 
of the geometrical condition requires more that 3 particles for the activation of interaction. This 
two models form densely-packed clusters.

Global interaction models CS$_q$ and CS do not form clusters (or more precisely the entire 
ensemble is a single cluster). Interestingly, for this particular random initial condition, the 
relatively small number of closest neighbours $q=3$ is enough to propagate interactions throughout 
the entire ensemble.

Moreover, in Fig.~\ref{fig:case1v} the maximal relative velocity between the particles $V(t)$, 
c.f. \eqref{V}, is presented, showcasing regions of sharp decay for the DI model.


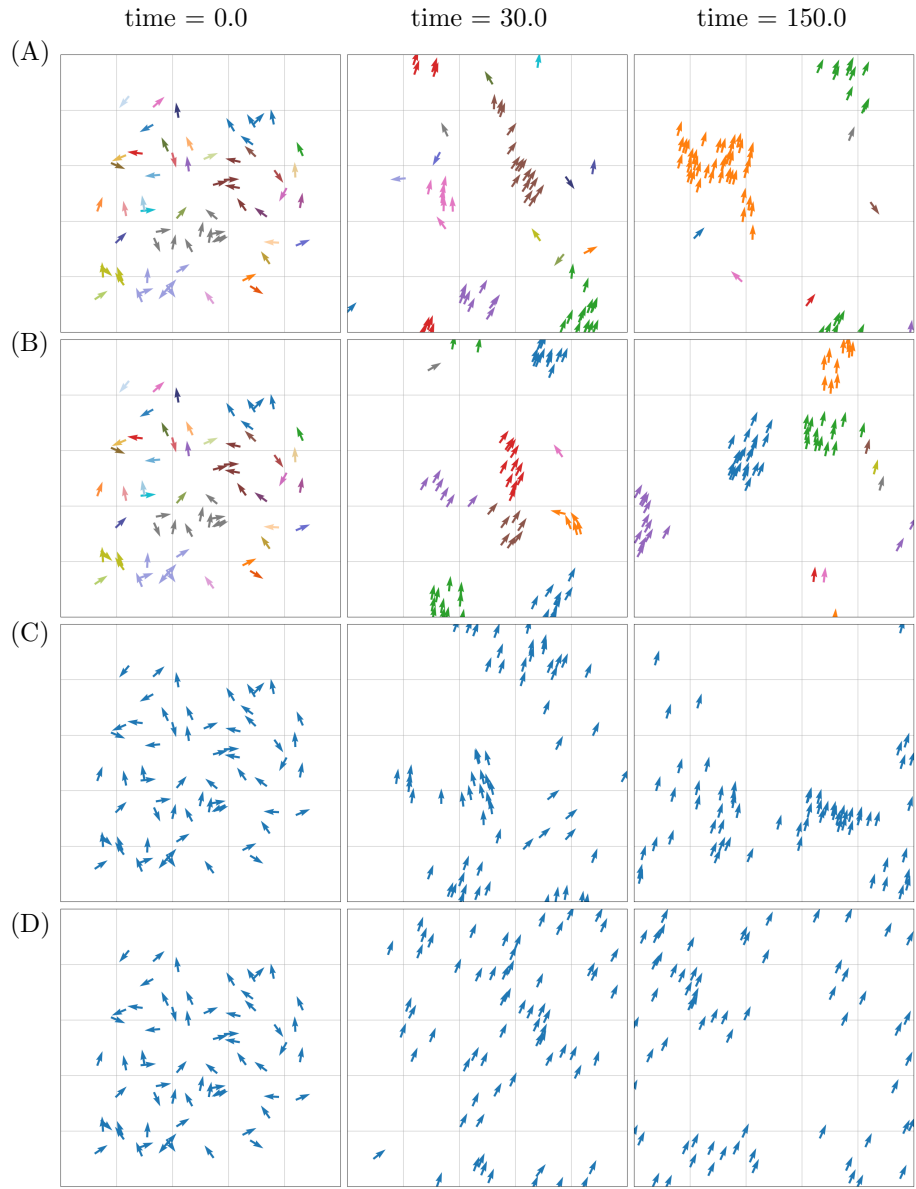
\begin{figure}[h!]
    \setlength{\unitlength}{.1\textwidth}
    \begin{picture}(10,12)
    \put(.5,0){\includegraphics[width=.9\textwidth]{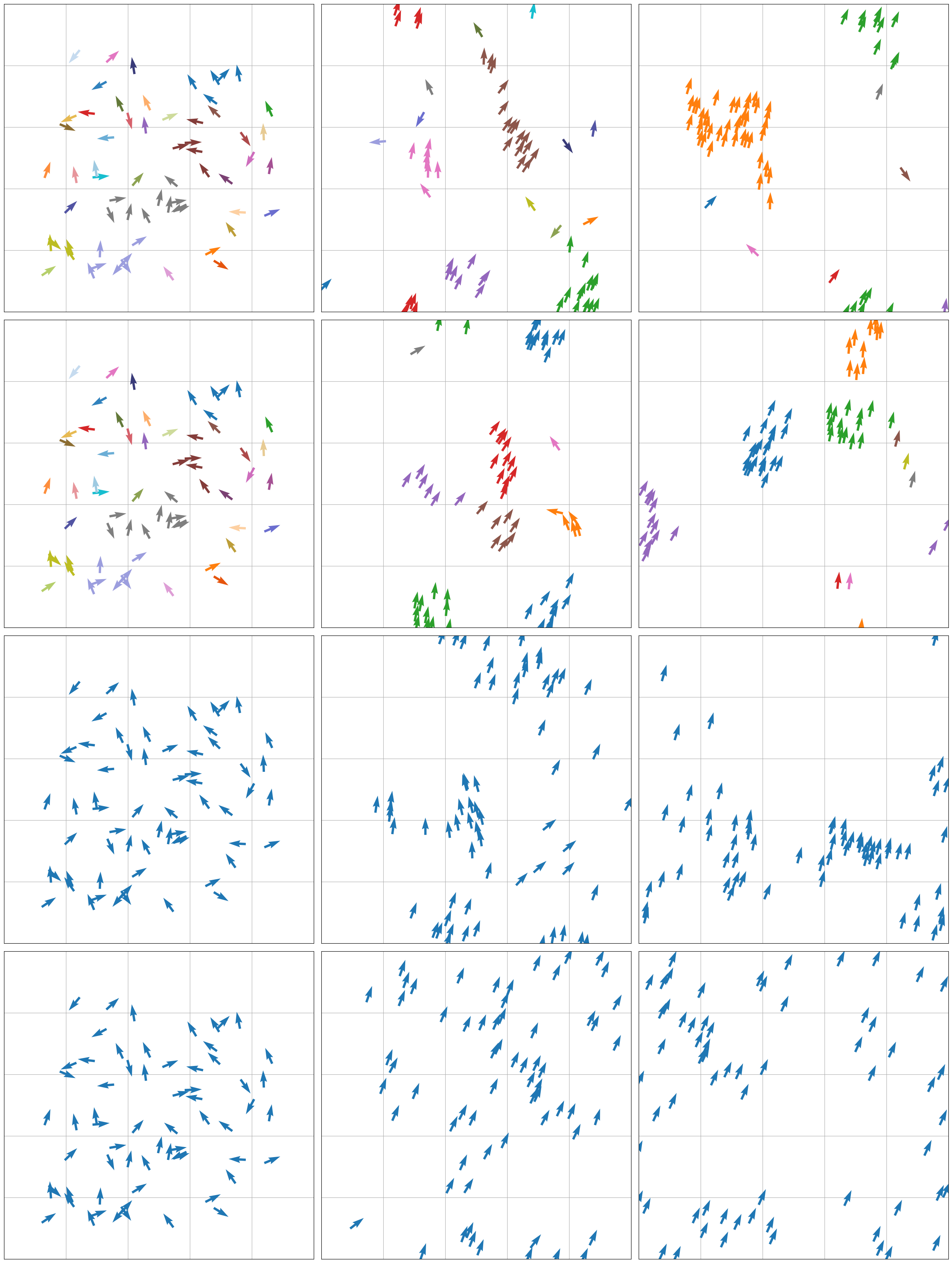}}
    \put(0,11.85){(A)}
    \put(0,8.8){(B)}
    \put(0,5.75){(C)}
    \put(0,2.7){(D)}
  
    \put(1.2,12.2){time = 0.0}
    \put(4.2,12.2){time = 30.0}
    \put(7.2,12.2){time = 150.0}
    
    \end{picture}
    \caption{Spontaneous behaviour of particles for DI(A), CS$_\delta$(B), CS$_q$(C), and CS(D), see Tab.~\ref{tab:case1}.}
    \label{fig:case1}
\end{figure}

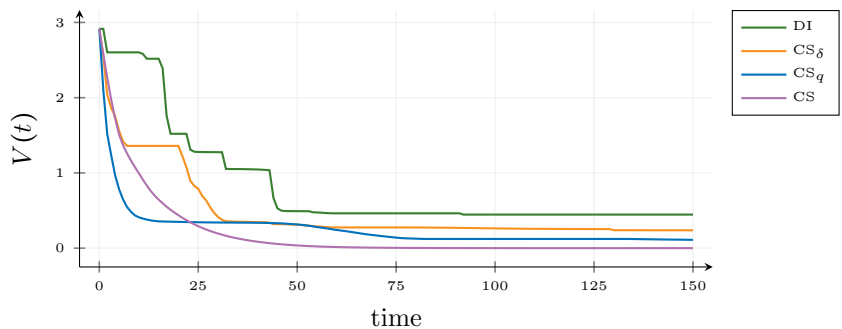
\begin{figure}[h!]
    \begin{center}
    \begin{tikzpicture}
    \begin{axis}[
    title = {},
    width = 10cm,
    height = 5cm,
    axis x line=bottom, axis y line=left,
    enlarge x limits={abs=5},
    enlarge y limits={abs=0.25},
    xlabel={time},
    ylabel={$V(t)$},
    tickpos=left,
    tick style = {semithick},
    grid=both,
    grid style={line width=.5pt, draw=gray!10},
    legend pos= outer north east,
    legend cell align={left},
    xtick={0.0, 25,  50, 75, 100, 125, 150},
    yticklabel style = {font=\tiny},
    xticklabel style = {font=\tiny},
    legend style={font=\tiny}
    ]
    \addplot[color= OliveGreen, line width=.8pt 
    ] table [x=time, y=V, col sep=tab] {./results/di.dat};
    \addlegendentry{ DI}
    \addplot[color= BurntOrange, line width=.8pt
    ] table [x=time, y=V, col sep=tab] {./results/csd.dat};
    \addlegendentry{ CS$_\delta  $}
    \addplot[color= RoyalBlue, line width=.8pt
    ] table [x=time, y=V, col sep=tab] {./results/csq.dat};
    \addlegendentry{ CS$_q  $}
    \addplot[color= Orchid, line width=.8pt
    ] table [x=time, y=V, col sep=tab] {./results/cs.dat};
    \addlegendentry{ CS   }
    \end{axis}
    \end{tikzpicture}
    \end{center}

   \caption{The maximal relative velocity between the particles for DI, CS$_\delta$, CS$_q$ and CS models.}
   \label{fig:case1v}
\end{figure}

\subsection{Sharp change of the total momentum}\label{sec:52}

Following Fig.~\ref{fig:2} we investigate the phenomenon of a single particle diverting an entire 
cluster and its dependence on the spatial distribution of particles within the cluster.
This scenario has been already addressed in the introduction. A cluster of $28$ particles moves to 
the right with velocity $(0.1,0)$, while a single particle moves in the opposite direction with 
velocity $(2.7,0)$, facing the cluster. Thus the total momentum (equal to the sum of velocities) of 
the particles is $0.1$.

We compare cases (A) and (B), c.f. Fig.~\ref{fig:2} and determine whether the total 
momentum changes its sign for the DI, 
CS$_\delta$, CS$_q$ and CS models with parameters $M$, $m$, $\delta$ and $q$ as in 
Tab.~\ref{tab:case1}. The results are showcased in Fig.~\ref{fig:case2v}. 
As expected, in the case of the CS model the total momentum is constant in both 
(A) and (B) cases. The total momentum for the 
CS$_\delta$ model changes (note that this is due to the non-symmetric normalization of the interaction 
$M=1/\#\cN_i$) but it remains positive. In the case of the CS$_q$ model the total momentum 
changes sign in both (A) and (B) cases. Finally, the DI model is the only one exhibiting a sharp 
distinction between (A) and (B). In (A) the total momentum changes its sign and in (B) -- it grows 
greater than for any other considered model.

\begin{figure}[h!]
    \begin{center}
        \begin{tikzpicture}
        \begin{axis}[
        title = {},
        width = 10cm,
        height = 5cm,
        axis x line=bottom, axis y line=left,
        enlarge x limits={abs=1},
        xlabel={time},
        ylabel={$\sum_i v_{i_x}(t)$},
        tickpos=left,
        tick style = {semithick},
        grid=both,
        grid style={line width=.5pt, draw=gray!10},
        legend pos= outer north east,
        legend cell align={left},
        yticklabel style = {font=\tiny},
        xticklabel style = {font=\tiny},
        legend style={font=\tiny}
        ]
              \addplot[color= Orchid, line width=.8pt, dotted
      ] table [x=time, y expr={28*\thisrow{avgx} }, col sep=tab] {./results/DIa.dat};
      \addlegendentry{ DI (A)}
      \addplot[color= Orchid, line width=.8pt, dash pattern=on 3pt off 9pt,
      ] table [x=time, y expr={28*\thisrow{avgx} }, col sep=tab] {./results/DIb.dat};
      \addlegendentry{ DI (B)   }
      
      \addplot[color= OliveGreen, line width=.8pt, dotted
      ] table [x=time, y expr={28*\thisrow{avgx} }, col sep=tab] {./results/CSda.dat};
      \addlegendentry{ CS$_\delta$ (A)}
      \addplot[color= OliveGreen, line width=.8pt, dash pattern=on 3pt off 9pt,
      ] table [x=time, y expr={28*\thisrow{avgx} }, col sep=tab] {./results/CSdb.dat};
      \addlegendentry{ CS$_\delta$ (B) }
      
      \addplot[color= BurntOrange, line width=.8pt, dotted 
      ] table [x=time, y expr={28*\thisrow{avgx} }, col sep=tab] {./results/CSqa.dat};
      \addlegendentry{ CS$_q$ (A)}
      \addplot[color= BurntOrange, line width=.8pt, dash pattern=on 3pt off 9pt,
      ] table [x=time, y expr={28*\thisrow{avgx} }, col sep=tab] {./results/CSqb.dat};
      \addlegendentry{ CS$_q$ (B) }
      
      \addplot[color= RoyalBlue, line width=.8pt, dotted
      ] table [x=time, y expr={28*\thisrow{avgx} }, col sep=tab] {./results/CSa.dat};
      \addlegendentry{ CS  (A)}
      \addplot[color= RoyalBlue, line width=.8pt, dash pattern=on 3pt off 9pt,
      ] table [x=time, y expr={28*\thisrow{avgx} }, col sep=tab] {./results/CSb.dat};
      \addlegendentry{ CS (B) }
      
        \end{axis}
        \end{tikzpicture}
    \end{center}
    
    \caption{Sum of horizontal velocities of particles for DI, CS$_\delta$, CS$_q$ and CS models.}
    \label{fig:case2v}
    
\end{figure}
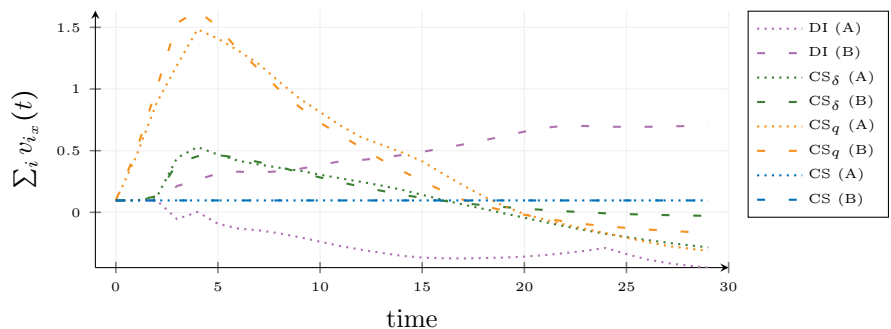

\subsection{Chain }\label{sec:case2}

In this scenario we present an interaction between a chain of particles and an individual. The chain 
consists of 21 particles that are equally distributed in the vertical direction. The initial 
velocity of the chain is small in comparison with the initial velocity of single particle and they 
are directed facing each other.

The initial velocity equals $v_0 = (0.1,0)$ and  $v_0 = (-8,0)$ for the chain and the individual, 
respectively. The DI model parameters $m=3$ and $M=1$, and we consider three scenarios that differs 
with parameter $\delta$ that equals $2$, $3$, $4$, cf. (A), (B), and (C) in 
Fig.~\ref{fig:case2} respectively. Since $\delta$ describes the interaction range we observe three 
different behaviours. In the case $\delta=2$, the cluster splits, since the single particle is 
causing activation of interaction, and due the velocity difference it takes cluster with itself.
For the second scenario, $\delta=3$, the chain deforms, but is able to turn back the singe agent. In 
the third example, with $\delta=4$, the chain is very stiff, and the single particle pushes the 
entire cluster.

\begin{figure}[h!]
    \setlength{\unitlength}{.1\textwidth}
    \begin{picture}(10,9)
    \put(.5,0){\includegraphics[width=.9\textwidth]{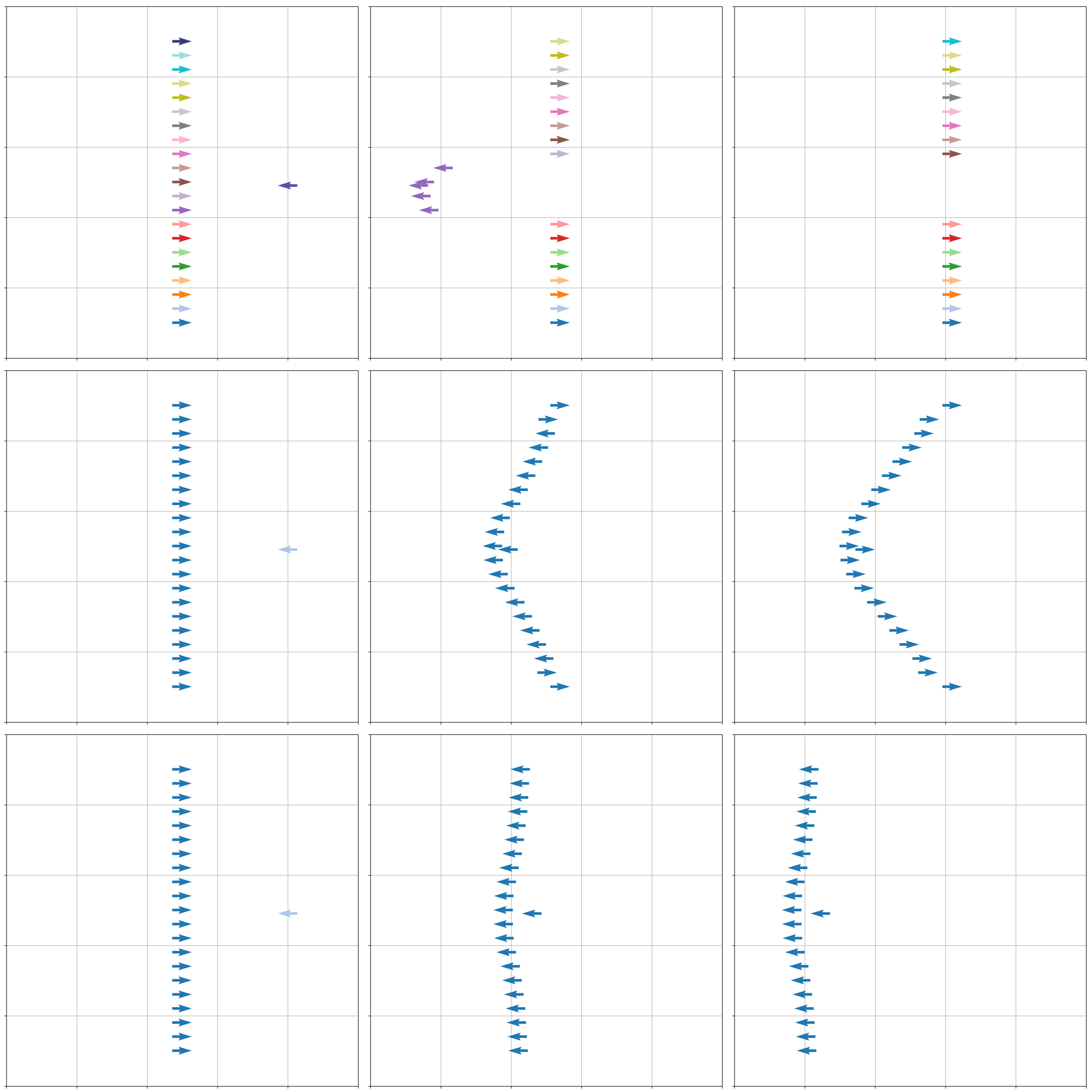}}
    \put(0,8.8){(A)}
    \put(0,5.75){(B)}
    \put(0,2.7){(C)}
    
    \put(1.2,9.2){time = 0.0}
    \put(4.2,9.2){time = 10.0}
    \put(7.2,9.2){time = 30.0}
    
    \end{picture}
    
    \caption{Interaction between chain and individual particle for DI$(8,\delta,1,3)$ model for $\delta=2$ (A), $\delta=3$ (B) and $\delta=4$ (C).}
    \label{fig:case2}
\end{figure}
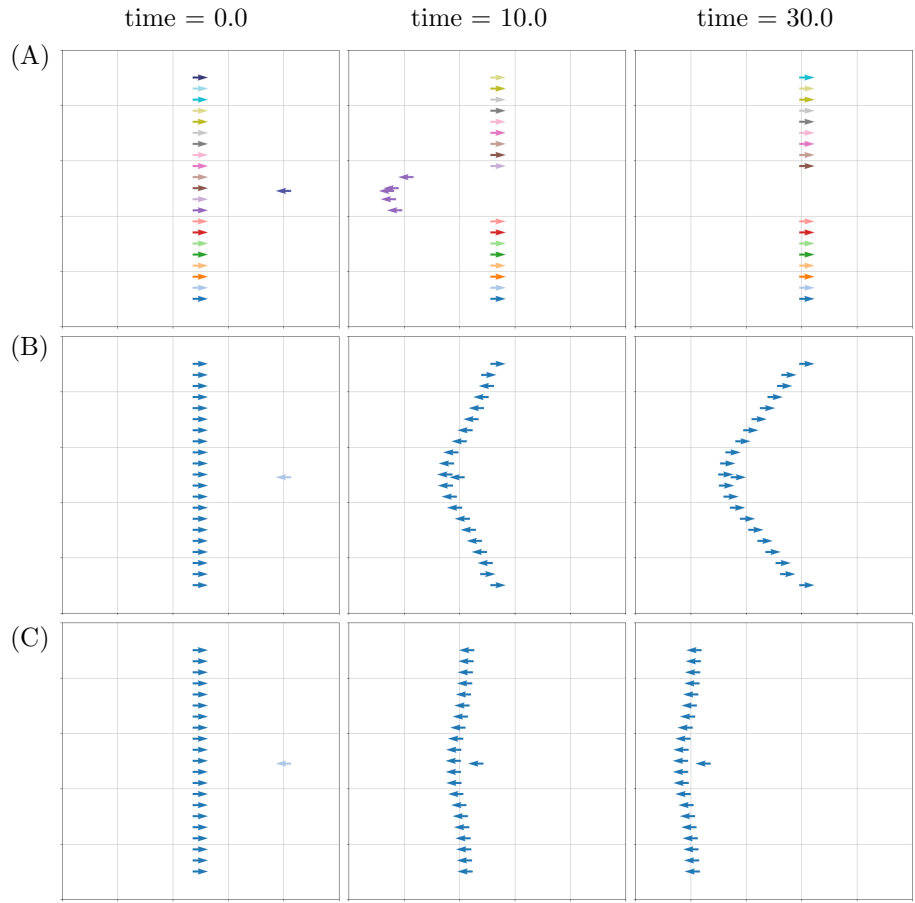

{\bf Supplement Materials.} Source code to reproduce simulations is available at~\cite{sourcecodeDI}.

\bigskip

{\bf Acknowledgement.} PM acknowledges the financial support by the Federal Ministry of
Education and Research of Germany, grant number 05M16NMA as well as
the GRK 2297 MathCoRe, funded by the Deutsche Forschungsgemeinschaft, 
grant number 314838170. PBM and JP have been partly supported by the Narodowe Centrum Nauki grant 
\textnumero 2018/30/M/ST1/00340 (HARMONIA). JP was additionally partially supported by the Narodowe 
Centrum Nauki grant \textnumero 2018/31/D/ST1/02313 (SONATA).

\bibliographystyle{abbrv}
\bibliography{l2}

\end{document}